\documentclass[a4paper, 10pt, twoside, notitlepage]{amsart}

\usepackage{amsmath,amscd}
\usepackage{amssymb}
\usepackage{amsthm}
\usepackage{comment}
\usepackage{graphicx, xcolor}

\usepackage{mathrsfs}
\usepackage[ocgcolorlinks, linkcolor=blue]{hyperref}

\usepackage[ocgcolorlinks,linkcolor=blue]{hyperref}

\newcommand{\LC}{\left(}
\newcommand{\RC}{\right)}

\theoremstyle{plain}
\newtheorem{thm}{Theorem}[section]
\newtheorem{prop}{Proposition}[section]

\newtheorem{rmk}[prop]{Remark}

\numberwithin{equation}{section}
\newcommand {\R} {\mathbb{R}} 
 \newcommand {\N} {\mathbb{N}}
 
\newcommand {\p} {\partial}

\newcommand{\eps}{\epsilon}
\newcommand{\vareps}{\varepsilon}

\newcommand{\wt}{\widetilde}

\newcommand{\norm}[1]{\lVert #1 \rVert}         

\DeclareMathOperator{\F} {\mathcal{F}}

 \definecolor{skyblue}{rgb}{0.85,0.85,1}

\author[]{Ru-Yu Lai}
\address{School of Mathematics, University of Minnesota, Minneapolis, MN 55455, USA}
\curraddr{}
\email{rylai@umn.edu}

\author[]{Yi-Hsuan Lin}
\address{Department of Applied Mathematis, National Chiao Tung University, Hsinchu 30050, Taiwan}
\curraddr{}
\email{yihsuanlin3@gmail.com}

\title[Inverse problems for fractional semilinear elliptic equations]{Inverse problems for fractional semilinear elliptic equations}

\begin{document}

\maketitle
\begin{abstract}
This paper is concerned with the forward and inverse problems for the fractional semilinear elliptic equation $(-\Delta)^s u +a(x,u)=0$ for $0<s<1$. 
For the forward problem, we proved the problem is well-posed and has a unique solution for small exterior data. The inverse problems we consider here consists of two cases. First we demonstrate that an unknown coefficient $a(x,u)$ can be uniquely determined from the knowledge of exterior measurements, known as the Dirichlet-to-Neumann map. 
Second, despite the presence of an unknown obstacle in the media, we show that the obstacle and the coefficient can be recovered concurrently from these measurements. Finally, we investigate that these two fractional inverse problems can also be solved by using a single measurement, and all results hold for any dimension $n\geq 1$.

	\medskip



\end{abstract}

 	\tableofcontents


\section{Introduction}
In this paper, we study inverse problems for fractional semilinear elliptic equations, which extends our earlier result \cite{lai2019global}. For $0<s<1$, let $\Omega \subset \R^n, n\geq 1$ be a connected bounded domain with $C^{1,1}$ boundary $\p \Omega$, and  $\Omega_e:=\R^n\setminus \overline{\Omega}$ be the exterior domain of $\Omega$. We study the inverse problem for the fractional semilinear elliptic equation:
\begin{align}\label{intro_eqn_S}
\begin{cases}
(-\Delta)^s u+ a(x,u) =0 & \hbox{ in } \Omega,\\
u=f  &  \hbox{ in } \Omega_e,\\
\end{cases}
\end{align}
where $a(x,u)$ is an unknown coefficient to be determined from the given information. We will characterize the regularity assumptions for $a(x,u)$ later.
Recall that the fractional Laplacian for $0<s<1$ is defined by 
\begin{align}\label{fractional Laplacian}
(-\Delta)^{s}u=c_{n,s}\mathrm{P.V.}\int_{\mathbb{R}^{n}}\dfrac{u(x)-u(y)}{|x-y|^{n+2s}}dy,
\end{align}
for $u\in H^s(\mathbb R^n)$, where P.V. denotes the principal value and  
\begin{equation}
c_{n,s}=\frac{\Gamma(\frac{n}{2}+s)}{|\Gamma(-s)|}\frac{4^{s}}{\pi^{n/2}}\label{c(n,s) constant}
\end{equation}
is a constant that was explicitly calculated in \cite{di2012hitchhiks}. Here $H^s(\R^n)$ is the standard fractional Sobolev space, which will be defined in Section \ref{Sec 2}.

For the coefficients $a(x,u)$, we assume that $a=a(x,z):\Omega\times \R \to \R$ satisfies 
\begin{align}\label{condition a}
\begin{cases}
a(x,0)=0,  \quad \p_z a(x,0)\geq 0, \\
\text{the map } \R \ni  z \mapsto  a(\cdot, z) \text{ is holomorphic with values in }C^s(\overline\Omega),
\end{cases}
\end{align}
where $C^s(\overline\Omega)$ stands for the H\"older space, which will be defined in Section \ref{Sec 2}.
In addition, it follows from \eqref{condition a} that $a$ can be expanded as a Taylor series 
\begin{align*}
a(x,z) = \displaystyle\sum^\infty_{k=1} a_k(x){z^k\over k!}, \qquad  a_k(x):=\p^k_z a(x,0)\in C^s(\overline\Omega),
\end{align*}
converging in $C^s(\Omega\times\R)$ space. Let us emphasize that the condition $\p_za(x,0)\geq 0$ in \eqref{condition a} plays an essential role to study related fractional inverse problems here, see for example Remarks~\ref{remark 2.2}.

The purpose of the paper consists of three main aspects. First, it attempts to prove the well-posed result for \eqref{intro_eqn_S} when small exterior data is imposed. Even though in \cite{lai2019global} the existence of the solution to \eqref{intro_eqn_S} is valid under certain conditions, but there is no guarantee for the uniqueness of the solution. 
With the establishment of well-posedness of \eqref{intro_eqn_S}, the second aspect is to study the inverse coefficient problem for the semilinear fractional equation. Thanks to the \emph{higher order linearization} technique, the Dirichlet and Neumann data can be measured on different open sets $W_1$ and $W_2$ in $\Omega_e$ (as indicated in Theorem~\ref{Main Thm 1}).
However, in \cite{lai2019global} where only first order linearization is performed, the condition $W_1=W_2\subset \Omega_e$ is crucial and inevitable for the study of the related inverse problem.  
Finally, in this paper, we also investigate simultaneous reconstruction of the unknown obstacle and unknown coefficients for the semilinear fractional equation, which is based on the higher order linearization approach as well.



We will prove the well-posedness of \eqref{intro_eqn_S} in Section \ref{Sec 2} that for any $f\in C^\infty_c(\Omega_e)$, whenever $\|f\|_{C^\infty_c(\Omega_e)}$ is sufficiently small, then \eqref{intro_eqn_S} is well-posed. It is worth to emphasizing that our regularity assumptions on the $C^{1,1}$ boundary is needed to prove the \emph{well-posedness} (see Section \ref{Sec 2}). 
 Under the well-posedness for the fractional semilinear elliptic equation \eqref{intro_eqn_S}, one can define the \emph{Dirichlet-to-Neumann} (DN) map via the bilinear form corresponding to the exterior problem \eqref{intro_eqn_S}, 
\[
\Lambda_{a}: H^s(\Omega_e)\to \big(H^{s}( \Omega_e)\big)^\ast, \qquad \Lambda_{a}(f)=\left.(-\Delta)^s u_f\right|_{\Omega_e}, 
\]
which maps $f$ to a nonlocal analogue of the Neumann boundary value of the solution $u_f$ to \eqref{intro_eqn_S}. Here $\big(H^{s}( \Omega_e)\big)^\ast$ is the dual space of $H^{s}( \Omega_e)$.

In the following two subsections, we first introduce the \emph{global uniqueness} in Section~\ref{subsec:uniqueness}, and then present the \emph{inverse obstacle problem} in Section~\ref{subsec:obstacle}, where there is an unknown obstacle encoded in the domain.  One of the main materials to prove of these theorems is the \emph{strong maximum principle}, which will be shown in Section \ref{Sec 2}.

\subsection{Global uniqueness}\label{subsec:uniqueness}
We investigate a fractional analogue of the Calder\'on problem for the fractional semilinear elliptic equation in any dimension. The main objection is to uniquely identify the unknown potential from the measurable data. In particular, due to the nonlocal nature, the inverse problem can be solved from \emph{partial data}, where only exterior Dirichlet and Neumann measurements in arbitrary open sets in $\Omega_e$ are needed. 

The fractional Calder\'on problem was first studied in the work \cite{ghosh2016calder}, for the linear case $a(x,u)=q(x)u$. Specifically, the authors in \cite{ghosh2016calder} showed that the unknown bounded potential $q$ is uniquely determined by the corresponding (nonlocal) DN map. Furthermore, there are extensive studies regarding fractional inverse problems, we refer readers to \cite{BGU18,CLL2017simultaneously,cekic2020calderon,ghosh2017calder,harrach2017nonlocal-monotonicity,harrach2020monotonicity,LLR2019calder,RS17,GRSU18,Ru17quantitative} for various settings with detailed discussions.

Let us state our first main result of the paper, which can be regarded as a nonlocal type Calder\'on inverse problem with partial data.

\begin{thm}\label{Main Thm 1}
	Let $\Omega \subset \R^n$, $n\geq 1$ be a bounded domain with $C^{1,1}$ boundary, and let $W_1,W_2\Subset \Omega_e$ be arbitrarily open subsets. Let $a_j(x,z)$ satisfy the condition \eqref{condition a} for $j=1,2$ and $0<s<1$. If the DN maps of the semilinear elliptic equations $(-\Delta)^s u + a_j (x,u)=0$ in $\Omega$ satisfy 
	\begin{align}\label{DN map in Thm 1}
	  \left.	\Lambda_{a_1}(f) \right|_{W_2} =  \left.	\Lambda_{a_2}(f) \right|_{W_2} \qquad \text{ for any }f\in C^\infty_c(W_1),
	\end{align}
	with $\norm{f}_{C^\infty_c(W_1)}<\delta$ for sufficiently small number $\delta>0$,
	then $$ a_1(x,z)=a_2(x,z) \qquad\hbox{ in } \Omega \times \R.$$
	\end{thm}

The proof of Theorem~\ref{Main Thm 1} is mainly based on the higher order linearization scheme. The idea of this method is to differentiate the original nonlinear equation with respect to small parameters multiple times in order to get a simpler linear equation.
The nonlinearity in some situations could in fact be beneficial to solve the nonlinear analogues of some unsolved inverse problems for linear equations. 	
In \cite{KLU2018}, the method of utilizing multiple linearization was introduced for nonlinear hyperbolic equations, see also \cite{CLOP,LUW2018}.
For the  semilinear elliptic equation, the uniqueness result with full data was proved in \cite{FO19,LLLS2019nonlinear,LLLS2019partial}, and was also relied on the higher order linearization.
Moreover, in the partial data setting, the same result was investigated by \cite{KU2019remark,LLLS2019partial}
under the condition of the coefficient $a(x,z)$, that is, $\p_z a(x,0)=0$. This constraint is particularly crucial there since it enables the use of the density result shown in \cite{ferreira2009linearized} that the set of the products of harmonic functions which vanish on a closed proper subset of the boundary is dense in $L^1(\Omega)$. An extension of this density result, where the products of harmonic functions are replaced by the scalar products of gradients of such functions, was shown in \cite{KU2019partial} recently.

In addition to the higher order linearization scheme, the proof of Theorem~\ref{Main Thm 1} also relies on its nonlocal analogue of density result for fractional equation. In particular, we show that the set of products of solutions to the fractional Schr\"odinger equation is dense in $L^\infty(\Omega)$, see Section~\ref{Sec 3} for detailed discussions of this density result.


\subsection{Inverse obstacle problem}\label{subsec:obstacle}
Inspired by the works \cite{CLL2017simultaneously,LLLS2019partial, KU2019partial}, we consider inverse problems for fractional semilinear elliptic equations in the presence of an unknown obstacle inside the media.

Let us introduce an inverse obstacle problem for fractional semilinear elliptic equations. Let $\Omega$ and $D$ be a bounded open sets with  $C^{1,1}$ boundaries $\p \Omega$ and $\p D$ such that $D\Subset\Omega$. Assume that $\p \Omega$ and $\Omega \setminus \overline{D}$ are connected. Let $a(x,z)$ satisfy \eqref{condition a} for $x\in \Omega \setminus \overline{D}$ and $z\in\R$. Consider the following fractional semilinear elliptic equation 
\begin{align}\label{main equation_cavity}
\begin{cases}
(-\Delta)^s u + a(x,u) =0 & \text{ in }\Omega \setminus \overline{D},\\
u =0 &  \text{ in } D, \\
u =f & \text{ in } \Omega_e.
\end{cases}
\end{align}
Here $f\in C^\infty_c(\Omega_e)$ with $\norm{f}_{C^\infty_c(\Omega_e)}< \delta$, where $\delta>0$ is sufficiently small. The well-posedness of \eqref{main equation_cavity} for small solutions will be shown in Section \ref{Sec 2} when the exterior data is sufficiently small, as a result, one can also define the corresponding DN map $\Lambda_{a}^D$, Neumann values $(-\Delta)^s u_f$ measured only in $\Omega_e$, by 
\[
\Lambda_{a}^D : H^{s}(\Omega_e) \to \big(H^{s}( \Omega_e)\big)^\ast,  \quad \Lambda_{a}^D: f \mapsto (-\Delta)^su_f |_{ \Omega_e},
\] 
where $u_f$ is the unique solution to \eqref{main equation_cavity}.
The (partial data) inverse obstacle problem is to determine the unknown obstacle $D$ and the coefficient $a$ simultaneously from the DN map $\Lambda_{a}^D$. Our second main result is as follows.

\begin{thm}[Simultaneous recovery: Unknown obstacle and coefficients]\label{Thm: Nonlinear nonlocal Schiffer's problem}
	Let $\Omega \subset \R^n$ be a bounded domain with connected $C^{1,1}$ boundary $\p \Omega$, $n\geq 1$ and $0<s<1$. Let $D_1, D_2\Subset \Omega$ be nonempty open subsets with $C^{1,1}$ boundaries such that $\Omega \setminus \overline{D_j}$ are connected. For $j=1,2$, let $a_j=a_j(x,z)$ satisfy \eqref{condition a} for $x\in \Omega\setminus \overline{D_j}$. Let $\Lambda_{a_j}^{D_j}$ be the DN maps of the following Dirichlet problems 
	\begin{align*}
	\begin{cases}
	(-\Delta)^s u_j +a_j (x,u_j)=0 & \text{ in }\Omega \setminus \overline{D_j}, \\
	u_j =0 & \text{ in } D_j,\\
	u_j =f & \text{ in } \Omega_e,
	\end{cases}
	\end{align*}
	with respect to the unique (small) solution $u_j$ for sufficiently small exterior data $f\in C^\infty_c(\Omega_e)$.
	Let $W_1,W_2 \Subset \Omega_e$ be arbitrarily open subsets.
    If
	\begin{align}\label{DN map in Thm 2}
	\left.\Lambda_{a_1}^{D_1}(f)\right|_{W_2}= \left.\Lambda_{a_2}^{D_2}(f) \right|_{W_2}  \qquad \text{ for any }f\in C^\infty_c(W_1),
	\end{align}
	with $\norm{f}_{C^\infty_c(W_1)}<\delta$, for sufficiently small number $\delta>0$,
	then 
	\[
	D:=D_1 = D_2 \quad \text{ and } \quad a_1(x,z)=a_2(x,z) \qquad \text{ in }(\Omega\setminus \overline{D})\times \R.
	\]
\end{thm}

The proof of Theorem \ref{Thm: Nonlinear nonlocal Schiffer's problem} is also based on higher order linearizations and the density property for solutions of the fractional Laplacian. Note that for the linear case when $a(x,u)=q(x)u$ in \eqref{main equation_cavity}, similar results were investigated by \cite{CLL2017simultaneously}. 

The paper is organized as follows. In Section \ref{Sec 2} we prove the well-posedness for the fractional semilinear elliptic equation $(-\Delta)^s u + a(x,u)=0$ in $\Omega$ with sufficiently regular exterior data $f$, in an appropriate sense. Equipped with the well-posedness, we define the DN map via the energy integration associated with the equation. 
In Section \ref{Sec 3}, we show that the set of the products of solutions to the fractional Schr\"odinger equation is dense in $L^\infty(\Omega)$, and we then carry out Theorem \ref{Main Thm 1} by applying the higher order linearization. Next we prove Theorem \ref{Thm: Nonlinear nonlocal Schiffer's problem} in Section \ref{Sec 4}. Finally, with a single measurement, one can also determine the coefficient and the obstacle in Section~\ref{Sec 5}.

\section{Preliminaries}\label{Sec 2}

In this section, we introduce notations and well-posedness of the fractional semilinear elliptic equation \eqref{intro_eqn_S}.

\subsection{Function spaces}

Let us first recall the H\"older spaces. Let $U\subset\R^n$ be an open set, for a given $0<\alpha <1$, the space $C^{k,\alpha}(U)$ is defined by
\[
C^{k,\alpha}(U):=\left\{f:U\to \R:\ \norm{f}_{C^{k,\alpha}(U)}<\infty \right\},
\]
where 
\[
\norm{f}_{C^{k,\alpha}(U)}:=\sum_{|\beta|\leq k}\norm{\p ^\beta f}_{L^\infty(U)}+\sup_{x\neq y, \ x,y\in U}\sum_{|\beta|=k}\frac{|\p ^\beta f(x)-\p^\beta f(y)|}{|x-y|^\alpha}.
\]
Here $\beta=(\beta_1,\cdots,\beta_n)$ is a multi-index with $\beta_i \in \N \cup \{0\}$ and $|\beta|=\beta_1 +\cdots +\beta_n$.  
Furthermore, we also denote the space
\[
C_0^{k,\alpha}(U):=\text{closure of }C^\infty_c(U) \text{ in }C^{k,\alpha}(U),
\]
where $k\in \N\cup \{0\}$ and $0<\alpha <1$. When $k=0$, we simply denote $C^\alpha(U) \equiv C^{0,\alpha}(U)$.

We next define the fractional Sobolev space. For $0<s<1$, the fractional Sobolev space is $H^{s}(\mathbb{R}^{n}):=W^{s,2}(\mathbb{R}^{n})$,
which is the $L^{2}$ based Sobolev space with the norm 
\begin{equation}\notag
\|u\|_{H^{s}(\mathbb{R}^{n})}^{2}\\=\|u\|_{L^{2}(\mathbb{R}^{n})}^{2}+\|(-\Delta)^{s/2}u\|_{L^{2}(\mathbb{R}^{n})}^{2}.\label{eq:H^s norm}
\end{equation}
Moreover, by using the Parseval identity, the semi-norm $\|(-\Delta)^{s/2}u\|_{L^{2}(\mathbb{R}^{n})}^{2}$
can also be expressed as 
\[
\|(-\Delta)^{s/2}u\|_{L^{2}(\mathbb{R}^{n})}^{2}=\left((-\Delta)^{s}u,u\right)_{\mathbb{R}^{n}},
\]
where $(-\Delta)^s $ is the fractional Laplacian \eqref{fractional Laplacian}.

Next, let $U\subset \R^n $ be an open set and $a\in\mathbb{R}$,
we consider the following Sobolev spaces, 
\begin{align*}
H^{a}(U) & :=\left\{u|_{U}:\, u\in H^{a}(\mathbb{R}^{n})\right\},\\
\widetilde{H}^{a}(U) & :=\text{closure of \ensuremath{C_{c}^{\infty}(U)} in \ensuremath{H^{a}(\mathbb{R}^{n})}},\\
H_{0}^{a}(U) & :=\text{closure of \ensuremath{C_{c}^{\infty}(U)} in \ensuremath{H^{a}(U)}},
\end{align*}
and 
\[
H_{\overline{U}}^{a}:=\left\{u\in H^{a}(\mathbb{R}^{n}):\,\mathrm{supp}(u)\subset\overline{U}\right\}.
\]
The Sobolev space $H^{a}(U)$ is complete under the graph norm
\[
\|u\|_{H^{a}(U)}:=\inf\left\{ \|v\|_{H^{a}(\mathbb{R}^{n})}:\,v\in H^{a}(\mathbb{R}^{n})\mbox{ and }v|_{U}=u\right\} .
\]
It is known that $\widetilde{H}^{a}(U)\subsetneq H_{0}^{a}(U)$,
and $H_{\overline{U}}^{a}$ is a closed subspace of $H^{a}(\mathbb{R}^{n})$.
In addition, when $U$ is a Lipschitz domain, it is also known that the dual spaces can be expressed as 
\begin{align*}
	(H^s_{\overline{U}}(\R^n))^\ast = H^{-s}(U), \quad \text{ and }\quad (H^s(U))^\ast=H^{-s}_{\overline U}(\R^n).
\end{align*}
For more details of the fractional Sobolev spaces, we
refer to \cite{di2012hitchhiks,mclean2000strongly}.

\subsection{Well-posedness} 
Let $\Omega\subset \R^n$ be a bounded domain with $C^{1,1}$ boundary $\p \Omega$, $n\geq 1$, and $0<s<1$. We  consider the Dirichlet problem with exterior data 
\begin{align}\label{eqn:S}
\begin{cases}
 (-\Delta)^s u + a(x,u) =0 & \hbox{ in } \Omega,\\
u=f  &  \hbox{ in } \Omega_e,
\end{cases}
\end{align}
where $f\in C^\infty_c(\Omega_e)$ for some $0<s <1$.

We first recall the $L^\infty$-estimate for the solution to the fractional Schr\"odinger equation, which has been derived in \cite[Proposition 3.3]{lai2019global}.

\begin{prop}\label{Prop: L infty bound}
	Let $\Omega\subset \R^n$, $n\geq 1$ be a bounded domain with Lipschitz boundary $\p\Omega$, and $0<s<1$. Let $q(x)\in L^\infty(\Omega)$ be a nonnegative potential, $F\in L^\infty(\Omega)$, and $f\in L^\infty(\Omega_e)$. Let $u\in H^s(\R^n)$ be the unique solution of 
	\begin{align*}
		\begin{cases}
		(-\Delta)^s u + q(x)u=F & \text{ in }\Omega, \\
		u=f &\text{ in }\Omega_e.
		\end{cases}
	\end{align*} 
	Then the following $L^\infty$ estimate holds 
	\begin{align*}
		\|u\|_{L^\infty(\Omega)}\leq C\left( \norm{f}_{L^\infty(\Omega_e)}+ \norm{F}_{L^\infty(\Omega)}\right),
	\end{align*}
	for some constant $C>0$ independent of $u$, $f$, and $F$.
\end{prop}
 
The proof of the proposition can be found in \cite[Section 3]{lai2019global}. We next prove the well-posedness of \eqref{eqn:S}, and we show that the solution will belong to the $C^s(\overline{\Omega})$-H\"older space.


\begin{thm}[Well-posedness]\label{Thm:well posedness}
Let $\Omega\subset \R^n$, $n\geq 1$ be a bounded domain with $C^{1,1}$ boundary $\p\Omega$, and $0<s<1$.
Suppose that $a=a(x,z)\in C^s(\Omega\times \R)$.
Then there exists $\varepsilon>0$ such that when
\begin{align}\label{small boundary}
f\in \mathcal{X}:=\left\{f\in C^\infty_c(\Omega_e) :\ \norm{f}_{C^\infty_c(\Omega_e)}\leq \vareps \right\},
\end{align} 
the boundary value problem \eqref{eqn:S} has a unique solution $u$. Moreover, there exists a constant $C>0$, independent of $u$ and $f$, such that 
$$
    \|u\|_{C^{s}(\R^n)} \leq C \|f\|_{C^\infty_c(\Omega_e)}.
$$ 
\end{thm}

\begin{proof}
	Suppose that $\|f\|_{C^\infty_{c}(\Omega_e)}\leq \varepsilon$, for some small $\vareps>0$ to be determined later, and one extends $f$ to the whole space $\R^n$ by zero so that $\|f\|_{C^\infty_{c}(\R^n)}\leq \varepsilon$. We first consider the solution $u_0$ to the linear Dirichlet problem 
\begin{align}\label{eqn:u0}
\begin{cases}
(-\Delta)^s u_0 +\p_za(x,0)u_0=0 & \hbox{ in } \Omega,\\
u_0=f  &  \hbox{ in } \Omega_e.\\
\end{cases} 
\end{align}
Since $\p_za(x,0)\geq 0$, one can has that \eqref{eqn:u0} is well-posed (see \cite[Section 2]{ghosh2016calder}), i.e., there exists a unique solution $u_0 \in H^s(\R^n)$ of \eqref{eqn:u0}.
In addition, via Proposition \ref{Prop: L infty bound}, one can see that $u_0\in L^\infty(\Omega)$ such that 
\begin{align}\label{uniform bound of u_0}
	\norm{u_0}_{L^\infty(\Omega)}\leq C\norm{f}_{L^\infty(\Omega_e)},
\end{align} 
for some constant $C>0$ independent of $u_0$ and $f$.

Let $w_0:=u_0-f$, then $w_0\in \wt H^s(\Omega)$ is the solution of 
\begin{align}\label{eqn:w_0}
	(-\Delta)^s w_0 =-\p _z a(x,0)u_0-(-\Delta)^s f \quad \text{ in }\Omega. 
\end{align}
Notice that the first term $\p _z a(x,0)u_0$ on the right hand side of \eqref{eqn:w_0} is bounded since $\p_za(x,0)\in C^s(\Omega)$ and \eqref{uniform bound of u_0}. The second term $(-\Delta)^s f$ is also bounded since $\norm{(-\Delta)^sf}_{L^\infty(\R^n)}\leq  C\norm{f}_{C^\infty_c(\R^n)}$ for some constant $C>0$ independent of $f$.
Thus, we conclude that the right hand side of \eqref{eqn:w_0} is in $L^\infty(\R^n)$.

Now, by applying the optimal global H\"older regularity~\cite[Proposition~1.1]{ros2014dirichlet} to the equation \eqref{eqn:w_0} in the bounded $C^{1,1}$ domain $\Omega$, we have
$$
   \|w_0\|_{C^{s }(\R^n)}\leq C \left( \norm{u_0}_{L^\infty(\Omega)}+ \norm{f}_{C^\infty_c(\R^n)}\right),
$$
 for some constant $C>0$ independent of $u_0$ and $f$. Combining this with \eqref{uniform bound of u_0}, we obtain that 
 \begin{align}
 	\norm{u_0}_{C^s(\R^n)}\leq \|w_0\|_{C^{s }(\R^n)}+\|f\|_{C^{s }(\R^n)}\leq C\norm{f}_{C^\infty_c(\R^n)},
 \end{align}
 which shows that the solution $u_0$ of \eqref{eqn:u0} is exactly global H\"older continuous.

If $u$ is the solution to \eqref{eqn:S}, then we have $v:=u-u_0$ satisfies
\begin{align}\label{eqn:v}
\begin{cases}
(-\Delta)^s v +\p_za(x,0)v= G(v) & \hbox{ in } \Omega,\\
v=0  &  \hbox{ in } \Omega_e.\\
\end{cases}
\end{align}
where we define the operator $G$ by
\begin{align*}
 G(\phi):= \p_za(x,0)(u_0+\phi)-a(x,u_0 +\phi),
\end{align*}
then by the Taylor expansion, we have
\begin{align}\label{definition G}
   G(\phi)= A(x,u_0+\phi)(u_0+\phi)^2,  
\end{align}
where 
$$
A(x, u_0+\phi) := -{1\over 2} \int^1_0 \p_z^2 a(x,t(u_0+\phi)) (1-t)dt.
$$

Let us define the set 
$$\mathcal{M}=\left\{\phi\in C^{s}(\R^n):\ \phi|_{\Omega_e}=0,\ \|\phi\|_{C^{s}(\R^n)}\leq \delta \right\},$$
where $\delta>0$ will be determined later.  It is easy to see that $\mathcal{M}$ is a Banach space.
We first claim that for $\phi\in\mathcal{M}$, the function $G(\phi)\in C^s(\overline{\Omega})$.
To begin, let us use the condition \eqref{condition a} of $a=a(x,z)$ and the Cauchy's estimates, then the coefficients $a_k(x)$ of $a$ satisfy 
\begin{align}\label{Cauchy estimate}
\norm{a_k}_{C^{s }(\overline{\Omega})} \leq \frac{k!}{R^k}\sup_{|z|=R}\norm{a(\cdot, z)}_{C^{s }(\overline{\Omega})}, \qquad R>0.
\end{align}
Furthermore, recalling that the H\"older space $C^{s}(\overline{\Omega})$ is an algebra due to  
\begin{align}\label{Holder algebra}
\norm{u_0 \phi}_{C^{s }(\overline{\Omega})}\leq C\left(\norm{u_0}_{C^{s }(\overline{\Omega})}\norm{\phi}_{L^\infty(\Omega)}+\norm{u_0}_{L^\infty(\Omega)} \norm{\phi}_{C^{s }(\overline{\Omega})} \right),
\end{align}
for any $u_0,\phi\in C^{s }(\overline{\Omega})$ and for some constant $C>0$ independent of $u_0,\phi$ (see \cite[Theorem A.7]{hormander1976boundary}).
Use \eqref{Cauchy estimate} and \eqref{Holder algebra}, then one can obtain 
\begin{align}\label{pointwise estimate 1}
\left\| \frac{a_k}{k!}(u_0 +\phi)^k\right\|_{ C^{s }(\overline{\Omega})} \leq \frac{C^k}{R^k}\norm{u_0+\phi}^k_{ C^{s}(\overline{\Omega})}\sup_{|z|=R}\norm{a(\cdot,z)}_{ C^{s}(\overline{\Omega})},
\end{align}
for $k\in \N$ and $R>0$. If we choose the number $R=2C\norm{u_0+\phi}_{ C^{s }(\overline{\Omega})}>0$ such that the series 
$$
a(x,z)-\p_za(x,0)z=\displaystyle \sum_{k=2}^\infty a_k(x)\frac{z^k}{k!}\leq \displaystyle \sum_{k=2}^\infty {1\over 2^k}\sup_{|z|=R}\norm{a(\cdot,z)}_{ C^{s}(\overline{\Omega})}
$$ converges in $ C^{s }(\overline{\Omega})$, which infers $G(\phi)\in  C^{s }(\overline{\Omega})$ as desired.

Moreover, we derive the following result. Let $g\in C^s(\overline{\Omega})$
, then there exists a unique solution $\tilde{v}\in H^s(\R^n)$ to the source problem 
\begin{align}\label{zero boundary value problem}
	\begin{cases}
	(-\Delta)^s \tilde{v} +\p_za(x,0)\tilde{v}=g & \text{ in }\Omega, \\
	\tilde{v}=0 & \text{ in }\Omega_e,
	\end{cases}
\end{align}
due to $\p_za(x,0)\geq 0$, see \cite{ghosh2016calder}.
By Proposition~\ref{Prop: L infty bound} and~\cite[Proposition~1.1]{ros2014dirichlet}, one has
\begin{align}\label{Ros-Oton estimate}
\|\tilde{v}\|_{C^{s}(\R^n)}\leq C\|g\|_{L^\infty(\Omega)},
\end{align}
for some constant $C>0$ independent of $g$ and $\tilde{v}$. 
We can now denote by $$\mathcal L_s^{-1}: g\in C^s(\overline{\Omega}) \rightarrow \tilde{v}\in C^s(\R^n)$$ the solution operator to \eqref{zero boundary value problem}.

Then we need to show the continuous operator 
$$\mathcal{F}:=\mathcal{L}_s^{-1}\circ G$$ 
is a contraction map on $\mathcal{M}$. 
Once this is done, we can find a fixed point $v$ in the Banach space $\mathcal{M}$ by applying the contraction mapping principle and this fixed point $v$ turns out to be the solution of \eqref{eqn:v}.

To this end, we first claim $\F: \mathcal{M}\to \mathcal{M}$.
By \eqref{definition G}, \eqref{Ros-Oton estimate} and the Taylor expansion of $a$, we obtain
\begin{align}\label{F:M to M}
  \begin{split}
   \norm{\F(\phi)}_{C^{s}(\R^n)}
  &\leq C\|G(\phi)\|_{C^{s}(\overline{\Omega})}\\
  &= C  \|a(x,u_0+\phi)- \p_za(x,0)(u_0+\phi)-a(x,0)\|_{C^{s}(\overline{\Omega})} \\
  &\leq C \|u_0 +\phi\|^2_{C^{s}(\overline{\Omega})}\\
  & \leq C(\delta+\vareps)^2,
  \end{split}
\end{align}
for any $\phi \in \mathcal{M}$, where we have utilized the condition \eqref{condition a} that $a(x,0)=0$.
In addition, one can also obtain that  
\[
  \norm{\F(\phi)}_{C^{s}(\R^n)} \leq C (\varepsilon+\delta)^2< \delta,
\]
thus $\F$ maps $\mathcal{M}$ into itself for $\vareps$ and $\delta$ small enough.

We next want to show that $\F$ is a contraction mapping. By using similar tricks as before, from \eqref{Ros-Oton estimate} and the mean value theorem, there exists a constant $C>0$ such that 
\begin{align}\label{some estimate 1}
\begin{split}
&\hskip.45cm \norm{\F(\phi_1)-\F(\phi_2)}_{C^{s}(\R^n)}  \\ 
&=  \|(\mathcal{L}_s^{-1}\circ G)(\phi_1) -(\mathcal{L}_s^{-1}\circ G)(\phi_2)\|_{C^{s}(\R^n)}   \\
&\leq  C \|G(\phi_1) -G(\phi_2) \|_{C^{s}(\overline{\Omega})}  \\
&=  C \|A(x,u_0 +\phi_1) (u_0+\phi_1)^2  -A(x,u_0 +\phi_2) (u_0+\phi_2)^2  \|_{C^{s}(\overline{\Omega})} \\
& \leq  C\|A(x,u_0 +\phi_1) \|_{C^{s}(\overline{\Omega})} \|(u_0+\phi_1)^2 - (u_0 + \phi_2)^2\|_{C^{s}(\overline{\Omega})}  \\
&\quad + C\| A(x,u_0 +\phi_1) -A(x,u_0 +\phi_2)\|_{C^{s}(\overline{\Omega})} \|(u_0+\phi_2)^2\|_{C^{s}(\overline{\Omega})}  \\
&\leq  C\|A(x,u_0 +\phi_1)\|_{C^{s}(\overline{\Omega})}  \|2 u_0 +\phi_1+\phi_2 \|_{C^{s}(\overline{\Omega})} \|\phi_2 -\phi_1\|_{C^{s}(\R^n)} \\
	&\quad + C\| \p_z A\|_{C^{s}(\R^n)} \|\phi_2 -\phi_1\|_{C^{s}(\R^n)}\|u_0+\phi_2\|^2_{C^{s}(\overline{\Omega})} .
\end{split} 
\end{align}
It implies that
$$
    \norm{\F(\phi_1)-\F(\phi_2)}_{C^{s}(\R^n)}\leq C_0((\vareps+\delta)+ (\varepsilon +\delta)^2 ) \|\phi_2 -\phi_1\|_{C^{s}(\R^n)}
$$
for some constant $C_0>0$ independent of $\phi_1,\phi_2$, $\vareps$ and $\delta$.
Finally, we take $\varepsilon, \delta$ sufficiently small so that $C_0((\vareps+\delta)+ (\varepsilon +\delta)^2 )<1$,
which implies that $\F$ is a contraction mapping on $\mathcal{M}$.

Therefore, by applying the contraction mapping principle, there is a unique solution $v\in \mathcal{M}$ to the equation \eqref{eqn:v}, such that $v$ satisfies
\begin{align*} 
\|v\|_{C^{s}(\R^n)}\leq  C (\|u_0\|^2_{C^s(\overline\Omega)} + \|v\|^2_{C^{s}(\overline\Omega)})\leq  C\LC \varepsilon \|f\|_{C^\infty_c(\Omega_e)} + \delta\|v\|_{C^{s}(\overline\Omega)}  \RC 
\end{align*} 
due to \eqref{F:M to M}.
For $\delta$ small enough, we can then get that 
\begin{align*} 
\|v\|_{C^{s}(\R^n)}
\leq  C \varepsilon \|f\|_{C^\infty_c(\Omega_e)}.
\end{align*} 
Finally, we obtain the solution $u=u_0+v \in C^{s}(\R^n)$ to \eqref{eqn:S} and it satisfies
\begin{align*} 
\|u\|_{C^{s}(\R^n)}
\leq  C \varepsilon \|f\|_{C^\infty_c(\Omega_e)},
\end{align*} 
for some constant $C>0$ independent of $u$ and $f$. This completes the proof of the well-posedness.
\end{proof}

\begin{rmk}\label{remark 2.2}
	Regarding the well-posedness, we have the following several remarks.
	\begin{itemize}
	   
		\item[(a)] The condition \eqref{condition a} of $a=a(x,z)$ plays an essential role to demonstrate the global $C^s(\overline{\Omega})$-H\"older estimates for the solutions of fractional semilinear elliptic equations in the proof of Theorem~\ref{Thm:well posedness}. Specifically, the proof relies on the result showed in \cite{ros2014dirichlet} that the optimal regularity of the solution for the fractional Dirichlet problem is $ C^s(\overline{\Omega})$ up to the boundary, for $0<s<1$.

		\item[(b)]In Theorem~\ref{Thm:well posedness}, the condition $\p_z a(x,0)\geq 0$, stated in \eqref{condition a}, is required in order to apply Proposition~\ref{Prop: L infty bound} to obtain the regularity estimates for solutions of the linearized fractional Schr\"odinger equation
			\begin{align}\label{fractional Schrodinger}
			\begin{cases}
			(-\Delta)^s v + \p _za(x,0)v=0 & \text{ in }\Omega, \\
			v=f & \text{ in } \Omega_e.
			\end{cases}
			\end{align}	 
	   However, in the case $\p_za(x,0)\not \geq 0$, we are not aware of any existing results which can be utilized to reach the same regularity as in Proposition~\ref{Prop: L infty bound}. 
	   Meanwhile for the local case ($s=1$), by applying the Schauder estimate (for example, see \cite{gilbarg2015elliptic}), one can obtain solutions with higher regularity $C^{2,\alpha}$ for some $0<\alpha<1$, whenever $\p_za(\cdot,0)$ and $f$ are H\"older continuous.

		\item[(b)] As we mentioned above, one can only expect the optimal regularity of solutions to be $C^s(\overline{\Omega})$, therefore $\nabla u$ is not well-defined in the classical sense when $u\in C^s(\overline{\Omega})$ for $0<s<1$. This turns out to be a barrier to study related inverse problems for fractional semilinear elliptic equations with nonlinear gradient terms.
	\end{itemize}
\end{rmk}

\subsection{The DN map}

We study in this section the associated DN map for fractional semilinear elliptic equations $(-\Delta)^s u +a(x,u)=0$. To define the DN map, we 	analyze the solution $u$ more carefully as follows.

In the proof of Theorem \ref{Thm:well posedness}, we have shown that there exists a unique solution $u\in C^s(\R^n)$ when the exterior data $f$ is sufficiently small in an appropriate sense. 
Let $w=u-f$, where $u$ is the unique solution of \eqref{eqn:S} and $f\in C^\infty_c(\Omega_e)\subset C^\infty_c(\R^n)$. Then $w$ is the solution of 
\begin{align}\label{equation w}
\begin{cases}
(-\Delta)^s w +w=-a(x,u)+u-f-(-\Delta)^s f &\text{ in }\Omega,\\
w=0&\text{ in }\Omega_e.
\end{cases}
\end{align}
We multiply \eqref{equation w} by $w$ and then integrate over $\R^n$. Moreover, we use the fact that $u\in C^s(\R^n)$, $f\in C^\infty_c(\R^n)$ and $a(x,u)$ is bounded to derive that $w\in H^s(\R^n)$. Thus, $u = w+f \in H^s(\R^n)$. Then we are able to define the DN map rigorously in the following manner.

\begin{prop}[DN map]\label{prop:DNmap} 
	 Let $\Omega\subset \R^n$ be a bounded domain with $C^{1,1}$ boundary $\p \Omega$ for $n\geq 1$, $s\in(0,1)$ and let $a=a(x,z)$ satisfy the condition \eqref{condition a}.
    Define
	\begin{equation}\label{equivalent integration by parts}
	\left\langle \Lambda_{a}f,\varphi\right\rangle := \int_{\R^n}(-\Delta)^{s/2}u_f (-\Delta)^{s/2}\varphi\, dx + \int_{\Omega}a(x,u)\varphi \, dx,
	\end{equation}
	for $ f,\varphi\in C^\infty_c(\Omega_e)$.
	Here $u_f\in C^s(\R^n)\cap H^s(\R^n)$ is the unique (small) solution of \eqref{eqn:S} with the exterior data $f\in  C^\infty_c(\Omega_e)$. Then, 
	\[
	\Lambda_{a}:H^{s}(\Omega_e)\to \big( H^{s}(\Omega_e)\big)^{*},
	\]
	which is bounded, and 
	\begin{equation}
	\left.\Lambda_a(f)\right|_{\Omega_e}=\left. (-\Delta)^s u_f\right|_{\Omega_e}.
	\end{equation}
\end{prop}
\begin{proof}
	We first note that \eqref{equivalent integration by parts} is not a bilinear form, since $a(x,u)$ is nonlinear. As we computed before, the solution $u$ of \eqref{eqn:S} belongs to the fractional Sobolev space $H^s(\R^n)$ for $0<s<1$. Therefore, the Parseval identity yields that
	 	\begin{align*}
		 &\int_{\R^n}(-\Delta)^{s/2}u_f (-\Delta)^{s/2}\varphi \, dx + \int_{\Omega}a(x,u)\varphi \, dx\\
		 =& \int_{\R^n}(-\Delta)^{s}u_f \varphi \, dx + \int_{\Omega}a(x,u)\varphi \, dx \\
		 =& \int_{\Omega_e}(-\Delta)^{s}u_f \varphi \, dx,
	\end{align*}
	where we have utilized that 
	\[
	\int_{\R^n} (-\Delta)^{s}u_f \varphi\, dx = \int_{\Omega}(-\Delta)^{s}u_f \varphi \, dx +\int_{\Omega_e}(-\Delta)^{s}u_f \varphi\, dx
	\]
	clarified in \cite[Section 3]{ghosh2016calder} and the equation \eqref{eqn:S} of $u$.
	Since $\varphi$ can be arbitrarily chosen and by the duality argument. This finishes the proof.
\end{proof}

\subsection{Some results}
In the remaining of this section, let us state several results, which will be utilized in the proof of our main theorems in the following two sections.

\begin{prop}[Strong uniqueness]\label{Prop:strong uniqueness}
	Let $U$ be a nonempty open subset of $\R^n$, $n\geq 1$ and $0<s<1$.
	Let $v\in H^{r}(\R^n)$ be the function with $v=(-\Delta)^s v=0$
	in some open set $U$ of $\R^n$, for $r\in \R$. Then $v\equiv0$ in $\mathbb{R}^{n}$. 
\end{prop}

\begin{prop}[Runge approximation] \label{Prop Runge-approximation-property}	
	Let $\Omega \subset \R^n$ be a bounded domain with $C^1$ boundary $\p \Omega$, for $n\geq 1$, and $0<s<1$. Let $\Omega_1\subseteq\mathbb{R}^{n}$
	be an arbitrary open set containing $\Omega$ such that $\mathrm{int}(\Omega_1\backslash\overline{\Omega})\neq\emptyset$ and $q(x)\in L^\infty(\Omega)$.
	Then for any $g\in L^{2}(\Omega)$, for any $\varepsilon>0$, we can
	find a function $v_{\varepsilon}\in H^{s}(\mathbb{R}^{n})$ which solves
	\[
	 (-\Delta)^s v_{\varepsilon}+q v_{\varepsilon}=0 \qquad\mbox{ in }\Omega\quad \text{ and }\quad \mathrm{supp}(v_{\varepsilon})\subseteq\overline{\Omega_1}
	\]
	and 
	\[
	\|v_{\varepsilon}-g\|_{L^{2}(\Omega)}<\varepsilon.
	\]
\end{prop}

The proofs of Proposition \ref{Prop:strong uniqueness} and \ref{Prop Runge-approximation-property} can be found in \cite{ghosh2016calder}, so we skip the detailed arguments here.
 
We next recall the \emph{maximum principle} for the fractional Schr\"odinger equation, which was investigated by the authors' previous work \cite[Section 3]{lai2019global}.
\begin{prop}[Maximum principle]\label{Prop: max principle}
	Let $\Omega\subset \R^n$, $n\geq 1$ be a bounded domain with Lipschitz boundary $\p\Omega$, and $0<s<1$. Let $q(x)\in L^\infty(\Omega)$ be a nonnegative potential. Let $u\in H^s(\R^n)$ be the unique solution of 
	\begin{align*}
	\begin{cases}
	(-\Delta)^s u + q(x)u=F & \text{ in }\Omega, \\
	u=f &\text{ in }\Omega_e.
	\end{cases}
	\end{align*} 
	Suppose that $0\leq F\in L^\infty(\Omega)$ in $\Omega$ and $0\leq f \in L^\infty(\Omega_e)$ in $\Omega_e$. Then $u\geq 0$ in $\Omega$.
\end{prop}

On top of Proposition~\ref{Prop: max principle}, we are able to prove the \emph{strong maximum principle} for the fractional Schr\"odinger equation and the proof is inspired by \cite[Theorem 2.3.3]{bucur2016nonlocal}.
\begin{prop}[Strong maximum principle]\label{Prop: strong max principle}
	Let $\Omega\subset \R^n$, $n\geq 1$ be a bounded domain with Lipschitz boundary $\p\Omega$, and $0<s<1$. Let $q(x)\in L^\infty(\Omega)$ be a nonnegative potential. Let $u\in H^s(\R^n)$ be the unique solution of 
	\begin{align*}
	\begin{cases}
	(-\Delta)^s u + q(x)u=F & \text{ in }\Omega, \\
	u=f &\text{ in }\Omega_e.
	\end{cases}
	\end{align*} 
	Suppose that $0\leq F\in L^\infty(\Omega)$ in $\Omega$ and $0\leq f \in L^\infty(\Omega_e)$ with $f\not \equiv 0$ in $\Omega_e$. Then $u>0$ in $\Omega$. 
\end{prop}

\begin{proof}
	Via Proposition \ref{Prop: max principle}, we have $u\geq 0$ in $\R^n$. Thus, if $u$ is not positive in $\Omega$, there must exist a point $x_0\in \Omega $ such that $u(x_0)=0$. This implies that 
	\begin{align}\label{proof of SMP}
	\begin{split}
		0&\leq  F(x_0)=(-\Delta)^s u(x_0)+q(x_0)u(x_0) \\
	&=  \frac{c_{n,s}}{2}\int_{\R^n} \frac{2u(x_0)-u(x_0+y)-u(x_0-y)}{|y|^{n+2s}}\, dy \\
	&= - \frac{c_{n,s}}{2}\int_{\R^n} \frac{u(x_0+y)+u(x_0-y)}{|y|^{n+2s}}\, dy\\
    & \leq  0,
	\end{split}
	\end{align}
	where the constant $c_{n,s}>0$ defined in \eqref{c(n,s) constant}. Here we used both $u(x_0+z)$ and $u(x_0-y)$ are nonnegative in $\R^n$ in the last inequality of \eqref{proof of SMP}, and the identity in the second line of \eqref{proof of SMP} stands for the fractional Laplacian (see \cite[Section 3]{di2012hitchhiks}). Therefore, by using \eqref{proof of SMP}, we obtain $u\equiv 0$ in $\R^n$, which contradicts to the assumption that $f\geq 0$ with $f\not \equiv0$ in $\Omega_e$.
\end{proof}

\section{Proof of Theorem \ref{Main Thm 1}}\label{Sec 3}

In this section, we prove Theorem \ref{Main Thm 1} by using the higher order linearization. Before doing so, let us first derive the following useful density result.

\begin{prop}\label{Prop: Density of the products of fractional Laplacian}
	Let $\Omega \subset \R^n$ be a bounded domain with $C^{1,1}$ boundary, $n\geq 1$ and $0<s<1$. Let $q\in L^\infty(\Omega)$ be nonnegative potentials. Let $h\in L^\infty(\Omega)$, if 
	\begin{align}\label{id_density}
 h  v_1\cdots v_J =0\qquad\text{in }\Omega,
	\end{align}
   for any solution $v_j$ to the fractional Schr\"odinger equations 
	\begin{align}\label{fractional Schrodinger equation}
	((-\Delta)^s+q) v_j =0  \qquad\text{in }\Omega 
	\end{align}
	with $ v_j|_{\Omega_e}\in C^\infty_c (W_1)$ for $j=1,\cdots, J$.
	Then $h\equiv 0$ in $\Omega$.
\end{prop}
\begin{proof}
By choosing suitable exterior data $f_j\in C^\infty_c (W_1)$ of the solutions $v_j$ to the fractional Schr\"odinger equation \eqref{fractional Schrodinger equation} so that $f_j>0$ in some nonempty open subset of $W_1$. By the strong maximum principle (Proposition \ref{Prop: strong max principle}), we have that $v_j>0$ in $\Omega$ and thus $v_1\cdots v_J>0$.
Therefore, the function $h$ in \eqref{id_density} must vanish in $\Omega$.
\end{proof}

\begin{rmk} 
Regarding the condition $\p_za(x,0)\geq 0$ and the above density result, we make the following remarks.	
	\begin{itemize}
		
		\item[(a)] The sign condition $\p_za(x,0)\geq 0$ holds the key to the density result in Proposition~\ref{Prop: Density of the products of fractional Laplacian} due to the strong maximum principle (Proposition \ref{Prop: strong max principle}). 
		
		

		\item[(b)] In the local counterpart ($s=1$), the condition $\p_z a(x,0)=0$ is necessary to study the partial data problem of simultaneously recovering the obstacle and the coefficient, see for instance, \cite{KU201909, LLLS2019partial}. The reason is as follows. The problem is in general related to show that if
		\[
		\int_\Omega h uv \, dx=0
		\]
		for any solutions $u,v$  to the equation  $-\Delta u= -\Delta v=0$ in $\Omega$ with $u=v=0$ in a nonempty closed proper subset of $\p \Omega$, then can one conclude that $h=0$ in $\Omega$? This is valid due to the result in \cite{ferreira2009linearized}, where the authors showed that the set of products of such harmonic function $u$ and $v$ is dense in $L^1$ space and hence $h$ must vanish in $\Omega$.
		We would like to note that if $\p_z a(x,0)\neq0$, then this density result is still open: show that
		\[
		\int_\Omega h uv \, dx=0
		\]
		for any solutions $u,v$  to the Schr\"odinger equation $(-\Delta+q)u=(-\Delta+q)v=0$ in $\Omega$ with $u=v=0$ in a nonempty closed proper subset of $\p \Omega$,
		implies that $h=0$ in $\Omega$.
		
		\item[(c)] From $(a)$ and $(b)$, we have seen that the fractional Laplacian $(-\Delta)^s, 0<s<1$ has \emph{better} approximation property than the classical Laplacian $(-\Delta)$. On the other hand, regarding the regularity of their solutions, the classical Laplacian indeed has \emph{better} global regularity estimates  and, in particular, the main difference between solutions of $(-\Delta)$ and $(-\Delta)^s$ is their boundary regularity.
	
	\end{itemize} 
\end{rmk}



Moreover, we have the following observation.
\begin{prop}\label{prop:induction}
Let $\Omega \subset \R^n$ be a bounded domain with $C^{1,1}$ boundary, $n\geq 1$ and $0<s<1$.
Let $\eps$ be a small parameter and let $f \in C_c^\infty (W_1)$. Let $u_j$ be the solution to the exterior problem
\begin{align}\label{Dirichlet problem in Section 3}
\begin{cases}
(-\Delta)^s u_j + a_j (x, u_j)=0 & \text{ in }\Omega, \\
u_j = \eps  f & \text{ in }\Omega_e,
\end{cases}
\end{align}
for any $|\eps|$ sufficiently small for $j=1,2$. 
Suppose that 
\begin{align}\label{prop:a12}
\left.\Lambda_{a_1}(f)\right|_{W_2} = \left.\Lambda_{a_2}(f)\right|_{W_2}\qquad \text{ for all }f\in C^\infty_c(W_1).
\end{align}
Then 
\begin{equation}\label{prop:subinduction}
    u_1^{(k)}(x)= u_2^{(k)}(x)\qquad \text{ in }\R^n \qquad \text{ for all }k\in \N,
\end{equation}
where we used the following abbreviation
$$ \p^k_\eps u_j (x;\eps):=\frac{\p^{k} u_j}{\p \eps^k}(x;\eps),$$
and in particular, when $\eps=0$, we denote
$$
  u^{(k)}_j(x) := \p^k_\eps u_j (x;0).
$$
\end{prop}
\begin{proof}
We will show \eqref{prop:subinduction} from $k=1,2$ up to $N+1$, respectively. 
Let us begin with the case $k=1$ by applying the first order linearization.
We differentiate \eqref{Dirichlet problem in Section 3} with respect to $\epsilon$, then 
\begin{align}\label{Prop:equ2a1 in 1st example}
\begin{cases}
(-\Delta)^s \left(\p_{\eps}u_j\right) + \p_{z} a_j (x,u_j)\left(\p_{\eps}u_j\right)=0 & \text{ in }\Omega, \\
\p_{\eps }u_j = f  &\text{ in } \Omega_e.
\end{cases}
\end{align}
Setting $\epsilon=0$ in \eqref{Prop:equ2a1 in 1st example}, we then have that 
\begin{align}\label{Prop:equation of 1st linearization in thm 1}
\begin{cases}
(-\Delta)^s u_j^{(1)}  +\p_za_j(x,0)u_j^{(1)} =0 & \text{ in }\Omega,\\
u_j ^{(1)} =f &\text{ in }\Omega_e,
\end{cases}
\end{align}
Here we have used the well-posedness of \eqref{Dirichlet problem in Section 3} so that $u_j(x;0)\equiv 0$ in $\Omega$ so that $\p_{z }a_j(x,u_j(x;0))=\p_za_j(x,0)$ in $\Omega$. 

By the hypothesis \eqref{prop:a12}, we have $(-\Delta)^s u_1|_{W_2} = (-\Delta)^su_2|_{W_2}$ that implies 
$$(-\Delta)^s u_1^{(1)} =(-\Delta)^s u_2^{(1)} \qquad\text{ in }W_2.$$
Combining it with the boundary condition $u_1^{(1)}=u_2	^{(1)}=f $ in $\Omega_e$, by the strong uniqueness result in Proposition~\ref{Prop:strong uniqueness}, we obtain
\begin{align}\label{Prop:v_1 =v_2a1 Rn}
u^{(1)}:=u_1^{(1)}=u_2^{(1)}\qquad\text{ in } \R^n.
\end{align}

Next to show \eqref{prop:subinduction} is valid for $k=2$, let's differentiate \eqref{Dirichlet problem in Section 3} twice with respect to small parameters $\epsilon$, and then setting $\eps=0$ yields
	\begin{align}\label{equ 4a1 in 1st example}
	\begin{cases}
	(-\Delta)^s u_j^{(2)}+ \p_{z}a_j (x,0)u_j^{(2)}+ \p ^2_{z}a_j(x,0)u^{(1)} u^{(1)}=0 & \text{ in }\Omega, \\
	u_j^{(2)}=0 & \text{ in }\Omega_e,
	\end{cases}
	\end{align}
	where we used the fact that $u_j(x;0)\equiv 0$ in $\Omega$ for $j=1,2$.
Similarly, by applying the hypothesis \eqref{prop:a12} again, we can derive that 
$$(-\Delta)^s u_1^{(2)}=(-\Delta)^s u_2^{(2)}\qquad\text{ in }W_2.$$
Moreover, we have the boundary condition $u_1^{(2)}=u_2^{(2)} =0$ in $\Omega_e$, then we obtain
\begin{align}\label{Prop:v_1 =v_2a1 Rn}
u_1^{(2)}=u_2^{(2)} \qquad\text{ in } \R^n.
\end{align}
by applying Proposition~\ref{Prop:strong uniqueness} again.

So far we have shown the case from $k=1$ to $k=2$ in \eqref{prop:subinduction}. For general case, for any $N\in \N$ we can also show 
\eqref{prop:subinduction} holds with $k=N$ along the similar argument as above.
	
Performing $N^{th}$ linearization of \eqref{Dirichlet problem in Section 3} at $\eps=0$ yields
	\begin{align}\label{equ 7a1 in 1st example}
	\begin{split}
&	(-\Delta)^s u_j^{(N)} + \p_z a(x,0)u_j^{(N)} \\
	&\hskip1.5cm + R_{N-1}(u_j,a_j)  
    + \p_z^{N} a_j(x,0)\left( u^{(1)}\right)^N=0 \qquad\text{ in }\Omega, 
	\end{split}
	\end{align}
with boundary data 
\begin{align}\label{bdryK}
u_1^{(N)}=u_2^{(N)}=0 \qquad  \hbox{ in }\Omega_e,
\end{align}
where $R_{N-1}(u_j,a_j)$ stands for a polynomial consisting of the functions $\p_z^{\beta}a_j(x,0)$ for $2\leq\beta\leq N-1$ and $u_j^{(k)}(x)$ for all $1\leq k\leq N-1$.

Following the same argument as the case $k=1,2$ above, the same DN maps gives $(-\Delta)^su_1^{(N)}=(-\Delta)^su_2^{(N)}$ in $W_2$ and then combining it with their boundary data \eqref{bdryK}, the strong uniqueness principle yields $u_1^{(N)}=u_2^{(N)}$ in $\R^n$ for any $N\in \N$. This completes the proof.
\end{proof}

Equipped with Proposition \ref{Prop: Density of the products of fractional Laplacian} and Proposition~\ref{prop:induction}, we are ready to show Theorem \ref{Main Thm 1}.

\begin{proof}[Proof of Theorem \ref{Main Thm 1}]	

It suffices to prove that for any fixed $N\in\N$,
\begin{align}\label{claim of Theorem 1}
\p_z ^\beta a_1(x,0 )=\p _z^\beta a_2(x,0), \quad \text{ for any }\beta\leq N+1.
\end{align}

Let $\eps$ be a small parameter and let $f\in C_c^\infty (W_1)$ be a nontrivial function satisfying $f\geq 0$,
and $u_j$ be the solution to the exterior problem
\begin{align}\label{Dirichlet problem in Section 3_1}
\begin{cases}
(-\Delta)^s u_j + a_j (x, u_j)=0 & \text{ in }\Omega, \\
u_j =\eps f & \text{ in }\Omega_e,
\end{cases}
\end{align}
for any $|\eps|$ sufficiently small for $j=1,2$. By Proposition~\ref{prop:induction}, the knowledge of the DN map $\Lambda_a$ yields \begin{align}\label{thm:subinduction}
     u^{(k)}:= u_1^{(k)}(x)= u_2^{(k)}(x)\qquad \text{ in }\R^n\ \hbox{ for all }k\in\N.
\end{align}

We apply the induction argument to show the uniqueness of the coefficient $a$ below.
\vspace{3mm}

 $\bullet$ {\it $1^{st}$ step: $\beta=1$}. 

\vspace{3mm}

 \noindent To this end, we first prove that $\beta=1$, that is,
\begin{align*}
\p_za_1(x,0)=\p_za_2(x,0).
\end{align*}
Recall that $u_j^{(1)}$ is the solution to \eqref{Prop:equation of 1st linearization in thm 1} and from \eqref{thm:subinduction}, we have $ u^{(1)}:= u_1^{(1)}(x)= u_2^{(1)}(x)$. Subtracting \eqref{Prop:equation of 1st linearization in thm 1} with $j=1$ from \eqref{Prop:equation of 1st linearization in thm 1} with $j=2$, one has
$$
   	\LC\p_z a_2(x,0)-\p_za_1(x,0) \RC u^{(1)}=0 \qquad\hbox{ in }\Omega.
$$
Based on the suitable chosen boundary condition $f\geq 0$ and \eqref{condition a}, we can apply Proposition~\ref{Prop: Density of the products of fractional Laplacian} 
to derive that $\p_z a_1(x,0)=\p_za_2(x,0)$.
 
In order to make the idea of using induction argument clearer, we will prove that 
$\p_z^2 a_1(x,0)=\p_z^2 a_2(x,0)$ in detail as follows.

\vspace{3mm}

$\bullet$ {\it $2^{nd}$ step: $\beta=2$.} 	
\vspace{3mm}

\noindent  To show $\p_z^2a_1(x,0)=\p_z^2a_2(x,0)$, as in the proof of Proposition~\ref{prop:induction}, let's differentiate \eqref{Dirichlet problem in Section 3_1} twice with respect to small parameters $\epsilon$, and then setting $\eps=0$ yields \eqref{equ 4a1 in 1st example}. 
In particular, by Proposition~\ref{prop:induction}, we know $u_1^{(2)}=u_2^{(2)}$ in $\R^n$, together with $\p_za_1(x,0)=\p_za_2(x,0)$ obtained in $1^{st}$ step, then \eqref{equ 4a1 in 1st example} implies
$$
    \left(\p^2 _za_2(x,0)-\p^2 _z a_1(x,0)\right)  u^{(1)}u^{(1)}=0  \qquad\hbox{ in }\Omega.
$$
where $u^{(1)}$ is the solution to \eqref{Prop:equation of 1st linearization in thm 1}. By Proposition~\ref{Prop: Density of the products of fractional Laplacian}, 
we obtain 
	$$\p^2 _za_2(x,0)=\p^2 _z a_1(x,0)\qquad \hbox{ in }\Omega. $$ 

\vspace{3mm}

$\bullet$ {\it $3^{rd}$ step: $\beta=N+1$.} 	

\vspace{3mm}

\noindent By the induction argument, for a fixed integer $1<N\in \N$, let's assume that 
\begin{equation}\label{inducition_assumption}
\p_z^\beta a_1(x,0)=\p_z^\beta a_2(x,0) \qquad \text{ for all }\beta=1,2, \cdots, N
\end{equation}
is valid. 
It remains to show this identity holds when $\beta=N+1$.

We then perform the linearization of order $N+1$ on \eqref{Dirichlet problem in Section 3_1} at $\eps=0$ yields 
\begin{align}\label{equ 7a1 in 1st examplea}
\begin{split}
&	(-\Delta)^s \left(u_j^{(N+1)}(x,0) \right) + \p_z a(x,0) u_j^{(N+1)}(x,0) \\
&\hskip1.5cm + R_{N}(u_j,a_j)  
+ \p_z^{N+1} a_j(x,0)\left( u^{(1)} \right)^{N+1}=0 \qquad\text{ in }\Omega,
\end{split}
\end{align}
where $R_{N}(u_j,a_j)$ stands for a polynomial consisting of the functions $\p_z^{\beta}a_j(x,0)$ for $2\leq\beta\leq N$ and $ u_j^{(k)}(x)$ for all $1\leq k\leq N$. Then $$R_{N}(u_1,a_1)=R_{N}(u_2,a_2)$$ due to \eqref{thm:subinduction} and \eqref{inducition_assumption}.

Following a similar argument as discussed in the second step above, 
by subtracting the equations \eqref{equ 7a1 in 1st examplea} with $j=1$ from the equation with $j=2$, then one has 
\begin{align*}
 \left(\p_z^{N+1} a_1(x,0)-\p_z^{N+1} a_2(x,0) \right) \left(u^{(1)} \right)^{N+1} =0,
\end{align*}
Finally, we use Proposition~\ref{Prop: Density of the products of fractional Laplacian} again to conclude that $\p_z^{N+1} a_1(x,0)=\p_z^{N+1} a_2(x,0) $ in $\Omega$, which proves the claim \eqref{claim of Theorem 1} for the case $\beta = N+1$. The last step to prove Theorem \ref{Main Thm 1} is via the condition \eqref{condition a} and the Taylor expansion of $a_1(x,z)$ and $a_2(x,z)$. This completes the proof.
\end{proof}

 
\section{Proof of Theorem \ref{Thm: Nonlinear nonlocal Schiffer's problem}}\label{Sec 4}
In this section we will show Theorem \ref{Thm: Nonlinear nonlocal Schiffer's problem}. The strategy is first to recover the obstacle $D$ from the first linearization of the equation
\[
(-\Delta)^su(x)+a(x,u)=0.
\]
Once $D$ is determined, the next step is to apply the similar arguments as in Section~\ref{Sec 3} to reconstruct the nonlinearity $a=a(x,z)$. 

\begin{proof}[Proof of Theorem \ref{Thm: Nonlinear nonlocal Schiffer's problem}]
	Let $f\in C^\infty_c(W_1)$ and $\eps>0$ small, then Theorem~\ref{Thm:well posedness} shows that there exists a unique solution $u_j(x)=u_j(x;\epsilon)$ in $C^s(\overline\Omega\setminus D_j)$ to the problem
	\begin{align}\label{Schiffer equation proof}
	\begin{cases}
	(-\Delta)^s u_j +a_j (x,u_j)=0 & \text{ in }\Omega \setminus \overline{D_j}, \\
	u_j =0 & \text{ in } D_j,\\
	u_j =\eps f & \text{ in } \Omega_e,
	\end{cases}
	\end{align}
	for $j=1,2$.
	
	\vspace{10pt}
	
	$\bullet$ {\it $1^{st}$ step: Recovering the obstacle.}\\
	
	\noindent  By differentiating \eqref{Schiffer equation proof} with respect to $\eps$ and setting $\eps=0$, one obtains 
	\begin{align}\label{one-cavity problem}
	\begin{cases}
	(-\Delta)^s u_j^{(1)}+\p_z a_j (x,0) u_j^{(1)}=0 & \text{ in }\Omega \setminus \overline{D_j},\\
    u_j^{(1)} =0 & \text{ in }D_j, \\
	u_j^{(1)} =f & \text{ in } \Omega_e,
	\end{cases}
	\end{align}
	where we use the same notation $u_j^{(1)}$ to denote 
	$$u_j^{(1)}(x):= \p_{\eps}\big|_{\epsilon =0}u_j $$
	and the fact that $u_j(x;0)\equiv 0 $ in $\Omega \setminus \overline{D_j}$ due to the well-posed result in Theorem~\ref{Thm:well posedness} for $j =1,2$.

	Making preparation to show $D_1=D_2$, we first prove that $u_1^{(1)}\equiv u_2^{(1)}$ in $\R^n$. 
From the DN map condition $\Lambda_{a_1}^{D_1}(f)=\Lambda_{a_2}^{D_2}(f)$ in $W_2\subset \Omega_e$, it yields that $(-\Delta)^s u_1 =(-\Delta)^s u_2$ in $W_2$ and then by performing the first linearization on these DN map, it leads to 
$(-\Delta)^s u_1^{(1)} =(-\Delta)^s u_2^{(1)}$ in $W_2$. Combining with the boundary conditions $u_1^{(1)}= u_2^{(2)}=f$ in $\Omega_e$, 
the strong uniqueness of the fractional Laplacian (see Proposition \ref{Prop:strong uniqueness}) implies that
\begin{align}\label{beta =1 in obstacle}
	u_1^{(1)}\equiv u_2^{(1)}\quad  \text{ in }\R^n.
\end{align}

Equipped with $u_1^{(1)}\equiv u_2^{(1)}$ in $\R^n$, we are ready to show $D_1=D_2$.
We apply a contradiction argument by assuming that $D_1 \neq D_2$.
Without	loss of generality, let us assume $f\neq 0$ and there exists a nonempty open subset $M\Subset D_{2}\backslash\overline{D_{1}}$.	
	By using $u_2^{(1)}=0$ in $D_2$ and the result showed above $u_1^{(1)}= u_2^{(1)}$ in $\R^{n}$, we get
	that 
	\begin{align}\label{thm2.1,v1}
	u_1^{(1)}= u_2^{(1)}=0\qquad \hbox{ in }M\Subset D_{2}\backslash\overline{D_{1}}. 
	\end{align} 
    By applying \eqref{thm2.1,v1} and the equation \eqref{one-cavity problem} with $j=1$, it is readily seen that 
	\begin{align}\label{thm2.1,v2}
    (-\Delta)^s u_1^{(1)} =-\p_z a_j(x,0) u_1^{(1)}=0 \qquad \hbox{ in }M\Subset \Omega\backslash\overline{D_{1}}. 
	\end{align}
	With \eqref{thm2.1,v1} and \eqref{thm2.1,v2}, the strong uniqueness property implies that $u_1^{(1)}\equiv0$
	in $\mathbb{R}^{n}$, which contradicts to the assumptions that $u_1^{(1)}=f\neq 0$ in $\Omega_e$. 
	Therefore, we obtain the uniqueness of the obstacle, namely,
	\begin{align*}
		D:=D_1=D_2 \Subset\Omega.
	\end{align*}
	A similar argument can be found in \cite{CLL2017simultaneously} for the linear fractional Schr\"odinger equation.
	
		\vspace{10pt}

	$\bullet$ {\it $2^{nd}$ step: Recovering the coefficient.}\\

	\noindent We note that the determination of the coefficient for the inverse obstacle problem here can be derived by following the same argument as in the proof of Theorem~\ref{Main Thm 1} with $\Omega$ replaced by $\Omega\setminus \overline{D}$. Instead of directly applying the proof of Theorem~\ref{Main Thm 1}, we take a slightly different approach to this problem by using the Runge approximation (Proposition~\ref{Prop Runge-approximation-property}) as follows. 

	To finish the proof, we only need to show the claim
	\begin{equation}\label{schiffer_claim}
	\p _z^\beta a_1(x,0) = \p _z^\beta a_2(x,0), \qquad \beta \in \N
	\end{equation}
	to hold. Let us proceed by applying the induction argument. For the case $\beta=1$, one can use the first linearization as in the first step, and substitute $D:=D_1=D_2$ into \eqref{one-cavity problem} to obtain  
	\begin{align*}
	\begin{cases}
	(-\Delta)^s u_j^{(1)}+\p_z a_j (x,0) u_j^{(1)}=0 & \text{ in }\Omega \setminus \overline{D},\\
	u_j^{(1)} =0 & \text{ in }D, \\
	u_j^{(1)} =f & \text{ in } \Omega_e.
	\end{cases}
	\end{align*}
	Next, let us apply the results from the global uniqueness result \cite[Section 5]{CLL2017simultaneously}, then we have 
	$\p_za(x,0):=\p_z a_1(x,0)=\p_za_2(x,0)$, for $x\in \Omega\setminus \overline{D}$, which proves the claim \eqref{schiffer_claim} as $\beta =1$.

	By induction, we assume that~\eqref{schiffer_claim} holds for $\beta\leq N$. Then we want to prove that \eqref{schiffer_claim} is valid for $\beta=N+1$. We recall that from Proposition~\ref{prop:induction}, we have
	\begin{align}\label{derivs up to N}
	u^{(k)}(x):=\p^{k}_{\eps}u_1(x;0)=\p^{k}_{\eps}u_2(x;0) \text{ in }\Omega\setminus \overline{D}, \quad \text{ for all }k=1,2, \cdots, N.
	\end{align}
	Even though Proposition~\ref{prop:induction} yields \eqref{derivs up to N} holds for $\beta\geq N+1$, for the argument below we do not need the information of $\p^{\beta}_{\eps}u_j$ with order $\beta\geq N+1$.

	Let us differentiate the equation~\eqref{Schiffer equation proof} $(N+1)$ times with respect to the small parameters $\eps$ for $j=1,2$, and subtract the resulting equation with $j=2$ from the one with $j=1$, one gets
	\begin{align}\label{equ 7 in 2nd example 1}
	\begin{split}
		&  (-\Delta)^s \LC u_1^{(N+1)} - u_2^{(N+1)} \RC + \p_z a(x,0) \left(u_1^{(N+1)} -u_2^{(N+1)} \right) \\
	& \qquad + \p_z^{N+1} \left(a_1(x,0)-a_2(x,0)\right)\left( u^{(1)}\right)^{N+1}=0 \qquad\text{ in }\Omega\setminus \overline{D}.
	\end{split}
	\end{align}
	Note that in the derivation of \eqref{equ 7 in 2nd example 1}, we used the assumption~\eqref{schiffer_claim} for $\beta \leq N$ and~\eqref{derivs up to N} to deduce that the terms with derivatives of order up to $N$ (that is, $R_{N}(u_1,a_1)-R_{N}(u_2,a_2)=0$) vanish in the subtraction. We also have the boundary data $u_1^{(N+1)} -u_2^{(N+1)}=0$ in $\Omega_e \cup  D$ and 
	\begin{align}\label{equ 7 in 2nd example}
	(-\Delta)^s u_1^{(N+1)} -(-\Delta)^s u_2^{(N+1)}=0\qquad\text{ in }W_2
	\end{align}
    via similar arguments as in Proposition \ref{prop:induction}.
    
	Note that the DN map only given in $\Omega_e$, but not in $ D$. Therefore integrating~\eqref{equ 7 in 2nd example} and using integration by parts would produce an unknown integral over $D$. Inspired by a proof in \cite[Theorem 1.2]{LLLS2019partial}, we need to compensate for the lack of the information in $D$ when performing an integration by parts. Let us consider a solution $v^{(0)}$ to the fractional Schr\"odinger equation 
	\begin{align}\label{harmonic_function_partial_data}
	\begin{cases}
	(-\Delta)^s v^{(0)}+\p _z a(x,0)v^{(0)}=0 & \text{ in }\Omega \setminus \overline{D}, \\
	v^{(0)}=0 & \text{ in } D, \\
	v^{(0)}=\eta &\text{ in } \Omega_e,
	\end{cases}
	\end{align}
	where $\eta\in C^\infty_c(W_2)$ with $\eta \geq 0$ and $\eta \not \equiv 0$.

	By the strong maximum principle of the fractional Laplacian (see Proposition \ref{Prop: strong max principle} for the case $q(x)=\p_za(x,0)\geq 0$) and by the preceding arguments for the positivity of solutions, we must have that $v^{(0)}>0$ in $\Omega \setminus \overline{D}$.
	Multiplying the equation \eqref{equ 7 in 2nd example} by this positive solution $v^{(0)}$, and then integrating the resulting equation, we have  
	\begin{align}\label{integral id for generalaa 1}
	\begin{split}
	0&=   \int_{\Omega_e \cup D}v^{(0)}(-\Delta)^s \left(u_2^{(N+1)} -u_1^{(N+1)} \right)dx \\
	&= \int_{\R^n}v^{(0)}(-\Delta)^s  \left(u_2^{(N+1)}-u_1^{(N+1)} \right)dx\\
	&\quad +\int_{\Omega}v^{(0)}\p_z a(x,0) \left(u_2^{(N+1)}-u_1^{(N+1)}\right) dx\\
	   &\quad + \int_{\Omega}  v^{(0)} \p_z^{N+1} \left(a_2(x,0)-a_1(x,0)\right)\left(u^{(1)}\right)^{N+1} dx,
	\end{split}
	\end{align}
	where we used $u_j^{(N+1)}$ is the solution to \eqref{equ 7 in 2nd example 1}.
	Due to the boundary data $ u_1^{(N+1)} = u_2^{(N+1)}=0$ in $D\cup \Omega_e$ and the equation $(-\Delta)^s v^{(0)}+\p_za(x,0)v^{(0)}=0$ in $\Omega\setminus \overline{D}$, the first two terms on the right hand side of \eqref{integral id for generalaa 1} become
	\begin{align*}
		&\hskip.45cm \int_{\R^n}v^{(0)}(-\Delta)^s  \left(u_2^{(N+1)}-u_1^{(N+1)}\right)dx \\
		&\quad +\int_{\Omega}v^{(0)}\p_z a(x,0)  \left(u_2^{(N+1)}-u_1^{(N+1)}\right) dx\\
		&= \int_{\R^n}  \left(u_2^{(N+1)}-u_1^{(N+1)}\right)(-\Delta)^sv^{(0)}\, dx \\
		&\quad +\int_{\Omega}v^{(0)}\p_z a(x,0) \left(u_2^{(N+1)}-u_1^{(N+1)}\right) dx\\
		&= \int_{\Omega_e \cup D}  \left(u_2^{(N+1)}-u_1^{(N+1)}\right)(-\Delta)^sv^{(0)}\, dx\\
		&= 0.
	\end{align*}
Thus, \eqref{integral id for generalaa 1} becomes
	\begin{align}\label{integral id for generalaa}
	\begin{split}
	0 = \int_{\Omega\setminus\overline{D}}\p_z^{N+1} \left( a_2(x,0)-a_1(x,0) \right)\left(u^{(1)} \right)^{N+1}v^{(0)}\,dx.
	\end{split}
	\end{align}

	Finally, with \eqref{integral id for generalaa}, by	applying \cite[Lemma 5.1]{CLL2017simultaneously} (an analogous version Runge approximation of Proposition~\ref{Prop Runge-approximation-property} in the domain $\Omega \setminus \overline{D}$), for any $g\in L^2(\Omega)$,
	there exists a sequence $(v^{(0)}_{m})_{m\in\N}$ in $H^s(\R^n)$ so that 
	\begin{align*}
	\begin{cases}
	(-\Delta)^s v^{(0)}_{m}+\p _z a(x,0) v^{(0)}_{m}=0  \text{ in }\Omega\setminus \overline{D},\\
	v^{(0)}_{m}=0 \text{ in }D,\\
	v^{(0)}_{m}   \ \ \hbox{has exterior values in $C^\infty_c(W_2)$},\\ 
	v^{(0)}_{m}=g +r^{(0)}_{m}, 
	\end{cases}
	\end{align*}
	where $r^{(0)}_{m}$ converges to $0$ in $L^2(\Omega\setminus \overline{D})$ as $m\to \infty$. Then $v^{(0)}_{m}$ converges to $g$ in $L^2(\Omega\setminus \overline{D})$ as $m\to \infty$. We substitute the solutions $v^{(0)}$ by $v^{(0)}_{m}$ into the integral identity \eqref{integral id for generalaa}, and then take the limit as $m\to \infty$ so that we have 
	\[
     \int_{\Omega\setminus \overline{D}}\LC\p_z^{N+1}  a_1(x,0)-\p_z^{N+1} a_2(x,0)\RC\left( u^{(1)}\right)^{N+1} g\,	dx =0\qquad\hbox{ in }\Omega\setminus \overline{D}.
	\]
Since $g$ is arbitrary, we further obtain the following identity
	\[
 \LC\p_z^{N+1}  a_1(x,0)-\p_z^{N+1} a_2(x,0)\RC\left( u^{(1)}\right)^{N+1}	=0\qquad\hbox{ in }\Omega\setminus \overline{D}.
\]
Likewise, by choosing the exterior data $f\geq 0$ in $W_1$ but $f\not \equiv 0$ in $W_1$ so that $u^{(1)}>0$ in $\Omega \setminus \overline{D}$, it leads to $\p_z^{N+1}  a_1(x,0)=\p_z^{N+1} a_2(x,0)$, see also Proposition~\ref{Prop: Density of the products of fractional Laplacian}. Thus the claim \eqref{schiffer_claim} follows by the induction argument. The last step to prove Theorem \ref{Thm: Nonlinear nonlocal Schiffer's problem} is via \eqref{schiffer_claim} and the Taylor expansion of $a_1(x,z)$ and $a_2(x,z)$ in $\Omega \setminus \overline{D}$. This finishes the proof.
\end{proof}


\section{Single measurement approach}\label{Sec 5}

In the reminder of this paper, we present a single measurement approach for both inverse coefficient problem (Theorem \ref{Main Thm 1}) and inverse obstacle problem (Theorem \ref{Thm: Nonlinear nonlocal Schiffer's problem}). 
Similar results for inverse problems for fractional equations were investigated by \cite{GRSU18} for the fractional Schr\"odinger equation with a single measurement, and \cite{cekic2020calderon} for the fractional Schr\"odinger equation with drift with finitely many measurements.

The uniqueness result obtained in this section, namely, 
\begin{align}\label{a single data}
    	a_1(x,u(x))=a_2(x,u(x)) \qquad \text{ for all $x$ in the domain},
\end{align}
indeed relies on much fewer data (one measurement) and a shorter argument than those presented in Section~\ref{Sec 3} and Section~\ref{Sec 4}. While in Theorem~\ref{Main Thm 1} and Theorem~\ref{Thm: Nonlinear nonlocal Schiffer's problem}, 
benefiting from the condition \eqref{condition a} and the application of the higher order linearization, one can further recover every single term $\p_z^ka(x,0)$ in the Taylor series of $a$ such that $a(x,z)$ is determined for all $z\in\R$. 
This could be viewed as more informative than \eqref{a single data} in the single measurement case.

Let us state the global uniqueness with one measurement as follows.

\begin{thm}[Global uniqueness with one measurement]\label{Main Thm 1 with one measure}
	Let $\Omega \subset \R^n$, $n\geq 1$ be a bounded domain with $C^{1,1}$ boundary, and let $W_1,W_2\Subset \Omega_e$ be arbitrarily open subsets. 
    Let $a_j(x,z)$ satisfy the condition \eqref{condition a} in $\Omega$, for $j=1,2$ and $0<s<1$. 
		Given any fixed function $f\in C^\infty_c(W_1)\setminus \{0\}$ such that $\norm{f}_{C^\infty_c(W_1)}<\delta$ for some sufficiently small number $\delta>0$, then $	\left. \Lambda_{a_1}(f)\right|_{W_2}=\left. \Lambda_{a_2}(f)\right|_{W_2}$ implies that 
	\begin{align*}
	a_1(x,u(x))=a_2(x,u(x)),\qquad \text{ for }x\in \Omega,
	\end{align*}
 where $u(x):=u_1(x)=u_2(x)$ and $u_j$ is the solution of $(-\Delta)^s u_j + a_j (x,u_j)=0$ in $\Omega$ with $u_j=f$ in $\Omega_e$, for $j=1,2$.
\end{thm}

\begin{proof}
	The proof is based on the strong uniqueness and the semilinear elliptic equation. Specifically, by using the condition $\left. \Lambda_{a_1}(f)\right|_{W_2}=\left. \Lambda_{a_2}(f)\right|_{W_2}$, we have that $(-\Delta)^s u_1 =(-\Delta)^s u_2$ in $W_2$. Moreover, from the boundary $u_1=u_2=f$ in $\Omega_e$, we further obtain 
	\[
	u_1-u_2=(-\Delta)^s (u_1 -u_2)=0\qquad \text{ in }W_2 \Subset \Omega_e.
	\]
	Thus, the strong uniqueness (Proposition \ref{Prop:strong uniqueness}) implies that 
	\begin{align}\label{unique solution with one measure}
		u_1=u_2 \qquad \text{ in }\R^n.
	\end{align} 
	Next by using \eqref{unique solution with one measure} and the fractional semilinear elliptic equations $(-\Delta)^s u + a_1 (x,u)=(-\Delta)^s u + a_2 (x,u)=0$ in $\Omega$, one concludes $a_1(x,u(x))=a_2(x,u(x))$ in $\Omega$, which completes the proof.
\end{proof}

Finally, let us prove the simultaneous reconstruction of the fractional inverse obstacle problem with a single measurement.

\begin{thm}[Simultaneous recovery with one measurement]\label{Thm:one data obstacle}
	Let $\Omega \subset \R^n$ be a bounded domain with connected $C^{1,1}$ boundary $\p \Omega$, $n\geq 1$ and $0<s<1$. Let $D_1, D_2\Subset \Omega$ be nonempty open subsets with $C^{1,1}$ boundaries such that $\Omega \setminus \overline{D_j}$ are connected, and let $W_1,W_2\Subset \Omega_e$ be arbitrarily open subsets. For $j=1,2$, let $a_j(x,z)$ satisfy the condition \eqref{condition a} in $\Omega\setminus \overline{D_j}$, for $j=1,2$ and $0<s<1$. We denote by $\Lambda_{a_j}^{D_j}$ the DN maps of the following Dirichlet problems 
	\begin{align}\label{equ of obstacle one measure}
	\begin{cases}
	(-\Delta)^s u_j +a_j (x,u_j)=0 & \text{ in }\Omega \setminus \overline{D_j}, \\
	u_j =0 & \text{ in } D_j,\\
	u_j =f & \text{ in } \Omega_e,
	\end{cases}
	\end{align}
    with respect to the unique (small) solution for sufficiently small exterior data $f\in C^\infty_c(\Omega_e)$.
    Given any fixed function $f\in C^\infty_c(W_1)\setminus \{0\}$ such that $\norm{f}_{C^\infty_c(W_1)}<\delta$ for some sufficiently small number $\delta>0$, then $	\left.\Lambda_{a_1}^{D_1}(f)\right|_{W_2}= \left.\Lambda_{a_2}^{D_2}(f) \right|_{W_2} $ implies that 
	\begin{align*}
	D:=D_1 = D_2 \quad \text{ and } \quad a_1(x,u(x))=a_2(x,u(x)) \qquad \text{ for }x\in \Omega\setminus \overline{D},
	\end{align*}
	where $u=u_1=u_2$ in $\Omega \setminus \overline{D}$.
\end{thm}

\begin{proof}
	Let us follow the two steps in the proof of Theorem \ref{Thm: Nonlinear nonlocal Schiffer's problem}.
	
	\vspace{10pt}
	
	$\bullet$ {\it $1^{st}$ step: Recovering the obstacle by one measurement.}\\
	
	\noindent Via the DN map condition $\Lambda_{a_1}^{D_1}(f)=\Lambda_{a_2}^{D_2}(f)$ in $W_2\subset \Omega_e$, it yields that $(-\Delta)^s u_1 =(-\Delta)^s u_2$ in $W_2$. Moreover, $u_1=u_2=f$ in $\Omega_e$, then we have 
	\[
	u_1-u_2=(-\Delta)^s (u_1-u_2)=0 \text{ in } W_2.
	\]
	The strong uniqueness (Proposition \ref{Prop:strong uniqueness}) implies that 
	\begin{align}\label{unique solution with 1 measure in obstacle}
		u_1=u_2 \quad \text{ in } \R^n.
	\end{align}
	
	Equipped with $u_1=u_2$ in $\R^n$, we are ready to show $D_1=D_2$, and we only need one nonzero $f\in C^\infty_c(W_2)$ to reconstruct the unknown obstacle.
	To this end, we apply a contradiction argument by assuming that $D_1 \neq D_2$.
	Without	loss of generality, let us assume $f\neq 0$ and there exists a nonempty open subset $M\Subset D_{2}\backslash\overline{D_{1}}$.	
	By using $u_2=0$ in $D_2$ and the result showed above $u_1=u_2$ in $\R^{n}$, we get
	that 
	\begin{align}\label{thm2.1,v1 single}
	u_1= u_2=0\qquad \hbox{ in }M\Subset D_{2}. 
	\end{align} 
   By using the fractional semilinear elliptic equation \eqref{equ of obstacle one measure}, one has 
	\begin{align}\label{thm2.1,v2 single}
	(-\Delta)^s u_1=-a_1(x,u_1)=-a_1(x,0)=0 \qquad \hbox{ in }M \Subset \Omega\backslash\overline{D_{1}},
	\end{align}
	where we have utilized the condition $a_1(x,0)=0$ in $\Omega$.
	With \eqref{thm2.1,v1 single} and \eqref{thm2.1,v2 single}, the strong uniqueness property implies that $u_1\equiv0$
	in $\mathbb{R}^{n}$, which contradicts to the assumptions that $u_1=f\neq 0$ in $\Omega_e$.
	Therefore, we obtain the uniqueness of the obstacle, namely,
	\begin{align*}
	D:=D_1=D_2 \Subset\Omega.
	\end{align*}

	\vspace{10pt}

	$\bullet$ {\it $2^{nd}$ step: Recovering the coefficient by one measurement.}\\
	
	As in the proof of Theorem \ref{Main Thm 1 with one measure}, by using \eqref{unique solution with 1 measure in obstacle} and the equation \eqref{equ of obstacle one measure} in the situation $D=D_1=D_2$, it can be easily seen that $(-\Delta)^s u_1 +a_1(x,u_1)=(-\Delta)^s u_2 +a_2(x,u_2)=0 $ in $\Omega\setminus \overline{D}$. This concludes that 
	\[
	a_1(x,u(x))=a_2(x,u(x)) \qquad \text{ for }x\in \Omega \setminus \overline{D},
	\]
	as desired.	
\end{proof}

\begin{rmk}
	In conclusion, in this section, we present substantially shorter and simpler arguments to prove the global uniqueness, namely, 
	$$
		a_1(x,u(x))=a_2(x,u(x)) \qquad \text{ for }x\in \Omega,
	$$
	and to show simultaneous recovery for fractional inverse obstacle problems. 	
	It is worthy noting that the same scenario does not appear in their local counterpart ($s=1$), see for example, \cite{LLLS2019partial}. The two main reasons are as follows.
	\begin{itemize}
		\item[(a)]  The reconstruction of the unknown obstacle in \cite{LLLS2019partial} relies on using the first order linearization that turns the nonlinear equation into a simpler linear one and, as a result, \emph{infinitely many} measurements are needed to make such linearization work. 
		
		\item[(b)] The fractional Laplacian has the strong uniqueness principle. This special feature makes the recovery of the solution $u$ without even knowing the coefficients possible, and therefore it largely simplifies the whole argument.
	\end{itemize}
\end{rmk}

\bigskip

\noindent\textbf{Acknowledgment.}
R.-Y. Lai is partially supported by the NSF grant DMS-1714490. Y.-H. Lin is partially  supported by the Ministry of Science and Technology Taiwan, under the Columbus Program: MOST-109-2636-M-009-006.

\bibliographystyle{alpha}
\bibliography{NSchrodingerRef}

\end{document}